\documentclass[a4paper,12pt]{article}

\usepackage[utf8]{inputenc}
\usepackage[a4paper,left=2cm,right=2cm,top=3cm,bottom=2.5cm]{geometry}
\usepackage{color}
\usepackage{graphicx, graphics, caption, subcaption}
\usepackage{lineno,hyperref}
\modulolinenumbers[5]
\usepackage{amssymb,amsmath,amsthm}
\usepackage{xpatch}
\usepackage{titlesec}

\titleformat*{\section}{\large\bfseries}
\titleformat*{\subsection}{\normalsize\bfseries}
\titleformat*{\paragraph}{\large\bfseries}
\titleformat*{\subparagraph}{\large\bfseries}

\def\dg{\mathrm{diag}\hspace{1mm}}

\def\1{\mathbf 1}
\def\mE{\mathbb{E}}
\def\mR{\mathbb{R}}
\def\mP{\mathbb{P}}
\def\tr{\mathrm{Tr}\hspace{0.8mm}}

\newtheorem{theorem}{Theorem}
\newtheorem{lemma}{Lemma}
\newtheorem{prop}{Proposition}

\newtheorem{cor}{Corollary}
\newtheorem{remark}{Remark}
\newtheorem{conjecture}{Conjecture}

\begin{document}
	\title{Low complexity convergence rate bounds for the synchronous gossip subclass of push-sum algorithms}
	\author{Bal\'azs Gerencs\'er%
		\thanks{B. Gerencs\'er is with the Alfr\'ed R\'enyi Institute of Mathematics, Budapest, Hungary and the E\"otv\"os Lor\'and University, Department of Probability and Statistics, Budapest, Hungary, {\tt\small gerencser.balazs@renyi.hu} He is supported by the J\'anos Bolyai Research Scholarship of the Hungarian Academy of Sciences.} %
		\footnotemark[3]%
		\and %
		Mikl\'os Kornyik%
		\thanks{M. Kornyik is with the Alfr\'ed R\'enyi Institute of Mathematics, Budapest, Hungary, {\tt\small kornyik.miklos@renyi.hu}} %
		\thanks{The research was supported by NRDI (National Research, Development and Innovation Office) grant KKP 137490.}
	}
	
	\maketitle

	\begin{abstract}
		We develop easily accessible quantities for bounding the almost sure exponential convergence rate of push-sum algorithms. We analyze the scenario of i.i.d.\ synchronous gossip, every agent communicating towards its single target at every step. Multiple bounding expressions are developed depending on the generality of the setup, all functions of the spectrum of the network. While the most general bound awaits further improvement, with more symmetries, close bounds can be established, as demonstrated by numerical simulations.
	\end{abstract}

	\section{Introduction}
	
	Average consensus algorithms have been around for a while  \cite{blondel2005convergence}, \cite{tsitsiklis:phd1984}, with the fundamental goal of computing the average of input values on a network in a distributed manner with only local communication and simple operations. Often some symmetry is imposed on the communication, in terms of the matrix describing the linear update of the vector of values to be either doubly stochastic, or even symmetric. This condition is quite well understood \cite{tahbaz2009consensus}, see the survey \cite{MAL-051} also for applications, further discussion and references.

	However, the interest for distributed averaging algorithms capable of handling asynchronous directed communications emerged, naturally driving away the representing update matrix from being doubly stochastic, still with the intent to compute the exact average. As a result, the successful scheme of \emph{push-sum} was proposed \cite{kempe2003gossip}, later also
	investigated under the name \emph{ratio consensus} \cite{hadjicostis2014average} and joined by variants such as \emph{weighted gossip} \cite{benezit2010weighted}.
	The goal of these algorithms is the same, 
	but now using only local, directed communication and without requiring message passing to happen synchronously or consistently across the network. Given the simple objective of the algorithm, it also serves as a building block for more complex tasks, e.g., the spectral analysis of the network \cite{kempe2008decentralized} or distributed optimization algorithms \cite{nedic2014distributed}.
	
	With other real-life communication challenges taken into account, the concept has been extended in multiple ways to handle such aspects, including packet loss \cite{hadjicostis2015robust} \cite{olshevsky2018fully}, delay \cite{hadjicostis2014average} or even the presence of malicious agents \cite{hadjicostis2022trustworthy}.
	In the meantime, there is work to better understand the effect of such communication deficiencies for the reference algorithms. The error of the consensus value compared to the true average for the push-sum algorithm has been analyzed in case of packet loss \cite{gerencser2018push}, similarly as it has been done for classic (linear) gossip \cite{frasca2013large}, \cite{olfati2004consensus}.
	
	An essential question in the analysis for usability and efficiency is understanding the asymptotics of the processes, their convergence and the rate at which it happens.
	In the cases above, the convergence of the push-sum algorithm (or variants) has been confirmed. Additionally, for the original push-sum scheme, an exponential convergence has been proven \cite{kempe2003gossip}. However, at that time the focus was not yet on approximating the true rate.
	An important step ahead was made in \cite{iutzeler2013analysis} providing a convincing upper bound along an unspecified, infinite subset of the timeline for the almost sure (a.s.) rate of convergence.
	More recently, the \emph{exact rate of a.s.\ convergence} has been identified \cite{gerencsr2019tight} for stationary ergodic updates as the spectral gap in terms of the Lyapunov exponents of random matrix updates with generous applicability. While being a clean representation with the concern that this Lyapunov spectral gap is known to be \emph{uncomputable} in general \cite{tsitsiklis1997lyapunov}. 
	As a follow-up, it was possible to combine the inspiration of \cite{iutzeler2013analysis} and the tool-set of \cite{gerencsr2019tight} to obtain an actual upper bound on the a.s.\ rate for the i.i.d.\ case \cite{gerencser2022computable}, now formulated by manipulating the Kronecker square of a single (random) update matrix, thus leading to a computable quantity. The bounds are solid, however for a graph on $N$ vertices, matrices of $N^2\times N^2$ have to be analyzed, quickly increasing in dimension.
	
	Our goal is to provide even simpler convergence rate estimates. For this purpose, we focus our attention to the natural setup, where a weighted network determines the communication scheme driving the consensus process. In particular, we assume synchronized gossip message passing, i.e. every node sending a single packet to a single (random) recipient at each time slot. Convergence of this scheme has been known since the formation of the push-sum concept \cite{kempe2003gossip}, ensuring that distributed average computation takes place.
	
	The bounds provided can be computed directly once the standard spectral description of the network is available. We are to formulate multiple variants, both to provide general, but more conservative estimates, and also sharper ones for a more restricted setting with stronger symmetries.
	
	The rest of the paper is structured as follows. In the next section we formally define the averaging process and state our results. Section \ref{sec:tools} builds a framework for the proof of the theorems, while Section \ref{sec:proofs} completes the proofs. Detailed numerical performance analysis and concluding remarks are provided in Section \ref{sec:sim} and Section \ref{sec:conc}. 
	
	\section{Main results}
	
	Let us introduce the push-sum algorithm along with other concepts that will be used.
	Given is a finite graph $G = (V,E)$ with the vertex set $V = [N] := \{1,2,\ldots,N\}$, having degree sequence $d_1,\ldots,d_N$. There is an initial vector of values $x(0)\in \mR^N$ at the vertices to be averaged. The process is also using an auxiliary vector initialized at $w(0) = \1 \in \mR^N$.
	
	At each time step, a linear row-stochastic update - representing local communication - is performed to both vectors as
	\begin{align*}
		x(t)^\top & = x(t-1)^\top K(t),\\
		w(t)^\top & = w(t-1)^\top K(t).
	\end{align*}
	The average $\bar{x}:=\frac 1N \sum_i x_i(0)$ is then locally estimated by $x_i(t)/w_i(t)$. 
	There is a wide generality of how $(K(t))_{t\geq 0}$ can be chosen. In the current paper, we focus on the scenario of i.i.d.\ $K(t)$, when at each step, every vertex sends a single message to a randomly chosen neighbor with a constant proportion and all these choices independent from one another. Formally, $K(t) \stackrel{d}{=} K = (1-q)I + q\sum_{i} e_ie_{\beta_i}^T$ with some fixed $q\in [0,1]$ and independent $\beta_i$. By setting $p_{ij}:=\mP(\beta_i=j)$, we obtain an overall transition probability matrix $P$ which by construction has to be compatible with the adjacency matrix of $G$.
	
	For convenience, we introduce the notation $P_q = (1-q)I + qP$.  It is easy to check that $\mathbb E K= P_q$. In case $P$ has only real eigenvalues let $\lambda_i$ denote its $i^{th}$ largest eigenvalue and let $\lambda_{q,i}=(1-q) + q\lambda_i$ denote that of $P_q$. Following our setup let us state our main results. Theorem \ref{thm:general_rate} targets scenarios in more general settings, while Theorem \ref{thm:special_rate} is designed for more symmetric cases.
	
\begin{theorem}
	\label{thm:general_rate}
	
	Let us consider a push-sum algorithm with message probability matrix $P$. Then
	\begin{equation}
		\label{eq:general_rate}
		\limsup \frac1t \max_i\log \left|\frac{x_i(t)}{w_i(t)}-\bar x\right| \leq \frac 12 \log\rho\bigg( (I-J) (P_q^\top P_q + q^2(\Gamma -P^\top P))(I-J)\bigg) \quad a.s.
	\end{equation}
	where $\Gamma$ is a diagonal matrix with $\gamma_{ii} = \sum_j p_{ji}$ and $\rho(\cdot)$ denotes the spectral radius.
	Furthermore, if $P$ is symmetric, then
	\begin{equation}
		\label{eq:general_rate_symmetric}
		\limsup \frac1t \max_i\log \left|\frac{x_i(t)}{w_i(t)}-\bar x \right| \leq \frac 12 \log ((1-q)^2 + 2q(1-q)\lambda_2 + q^2) \quad a.s.
	\end{equation}
\end{theorem}
\begin{remark}
	\label{cor:general_rate_graph}
	In case each vertex chooses a recipient uniformly among its neighbors, the bounding quantity in Theorem \ref{thm:general_rate} will depend only on the graph structure, furthermore the diagonal matrix $\Gamma$ takes the form $\Gamma_{ii}=\sum\limits_{j:(j,i)\in E}d_{j}^{-1}$ with $d_j$ denoting the degree of vertex $j$. 
	
\end{remark}
\begin{proof}
	It is easy to check that in this case $\mathbb E K = P_q$ with $ P=D^{-1} A$, where $D$ denotes the diagonal matrix consisting of the degrees of the underlying graph's vertices, while $A$ denotes the graph adjacency matrix.	
	
\end{proof}
A better bound can be obtained for cases with stronger symmetries. A message probability matrix is said to be transitive if for any pair $(i,j)$ there exists a permutation matrix $\Pi$ with $\Pi_{ij} = 1$ such that $ \Pi P \Pi^{-1} = P$.
\begin{theorem}
	\label{thm:special_rate}
	Suppose that the message probability matrix $P$ is symmetric and transitive. Then 

	\begin{equation}
		\label{eq:transitive_rate}
		\limsup \frac1t \max_i\log \left|\frac{x_i(t)}{w_i(t)}-\bar x\right| \leq \frac 12 \log\xi_1,
	\end{equation}
	with $\xi_1$ being the largest root of the polynomial

	$$ f(\xi) =  \prod_{i>1} (\xi-\lambda_{q,i}^2) - \frac{q^2}N \sum_{i>1} (1-\lambda_i^2)\prod_{i\neq j>1} (\xi-\lambda_{q,i}^2). $$
\end{theorem}
\begin{remark}
	If $G$ is a transitive graph and each vertex chooses a recipient uniformly among its neighbors, then the corresponding message probability matrix satisfies the assumption of Theorem \ref{thm:special_rate}.
\end{remark}
A special case has been analyzed in
\cite{iutzeler2013analysis}, where the underlying topology was given by the complete graph with $q=1/2$. For this topology and general $q$ Theorem \ref{thm:special_rate} immediately gives
\begin{align*}
	f(\xi) = (\xi - (1-q)^2)^{N-2} \big(\xi-(1-q)^2- q^2(1-N^{-1})\big).
\end{align*} 
from which the convergence rate bound $(1-q)^2 + q^2(1-N^{-1})$ can be easily obtained.

\section{Tools}
\label{sec:tools}

Let us first introduce a framework and corresponding tools in a general setting. The alignment to the assumptions of the theorems will be carried out later.

First we remark that the elements $w_i(t)$ are all positive because the diagonal elements of the nonnegative $K(t)$ are strictly positive.
Using the notations $H(t)=K(1)K(2)\cdots K(t)$, $J=\mathbf 1 \mathbf 1^\top /N$
easy calculation shows
\begin{align*}  x(t)^\top - \bar x w(t)^\top  &= x_0^\top H(t)- \bar x w(t)^\top \\ 
	&= x_0^\top(JH(t) + (I-J)H(t)) - \bar x w(t)^\top \\ 
	&=x_0^\top(I-J)H(t)
\end{align*}
meaning that 
\begin{align}
	\nonumber	\max_i \left|\frac{x_i(t)}{w_i(t)}-\bar x\right| &\leq C \{\min_i(w_i(t))\}^{-1} ||x_0||_2 ||(I-J)H(t)||_2\\ &\leq C \{\min_i(w_i(t))\}^{-1} ||x_0||_2 ||(I-J)H(t)||_F. \label{consensus}
\end{align}
from some constant $C>0$. We are interested in the almost sure convergence rate of the quantity on the left of \eqref{consensus}. As we will see the dominant term will be $||(I-J)H(t)||_F$. To get a handle on this factor let us analyze
the expectation of $ || (I-J)H(t)||_F^2 $. 
\begin{align*}
	\mathbb E||(I-J)H(t)||_F^2 &= \mathbb E\ \mathrm{Tr}\{(I-J)H(t)H(t)^\top (I-J)\} \\ & = \mathrm{Tr}\{ (I-J) \mathbb E[H(t)H(t)^\top] (I-J) \}.
\end{align*}
According to the definition of $H(t)$ we can write
\begin{align}
	\label{phi_patter}
	\mathbb E[H(t)H(t)^\top] = \mathbb E \big[\mathbb E[K(1) \tilde H(t) \tilde H(t)^\top K(1)^\top\big| \tilde H(t) ] \big] = \mathbb E[\mathbb E\big.  [K(1)XK(1)^\top]\big|_{X=\tilde H(t)\tilde H(t)^\top}],
\end{align}
where $\tilde H(t) = K(2)K(3)\cdots K(t)$, thus by the i.i.d.\ nature of the updates $\tilde H(t) \stackrel{d}{=} H(t-1)$.
This motivates the following definition of the linear operator $\Phi: \mathbb R^{N\times N} \to \mathbb R^{N\times N}$ acting on matrices:
$$
\Phi(X):=\mathbb E[KXK^\top],
$$
which we will need to understand for further developing \eqref{phi_patter}.
For satisfactory notation, before we progress let us introduce the linear operator $\Psi: \mathbb R^{N\times N} \to \mathbb R^{N}$ 
$$ (\Psi(X))_{i} =  x_{ii} $$
and its pseudo-inverse $\Psi^-:\mathbb R^{N}\to \mathbb R^{N\times N}$ 
$$ (\Psi^-(v))_{ij} = \begin{cases}
	v_{i} & \mbox{ if } i=j, \\
	0 & \mbox{ otherwise.}
\end{cases}  $$ 
Following the pattern of (\ref{phi_patter}) we can prove the following
\begin{lemma} For any matrix $X$, we have
	\begin{align*}
		\Phi(X) = P_qXP_q^\top +q^2\big\{\Psi^{-} [P\ \Psi(X)] - \Psi^- \Psi(PXP^\top) 	\big\}
	\end{align*}
	reminding that $P_q = (1-q)I + qP$.
\end{lemma}
\begin{proof}
	Let $L := \sum_i e_i e_{\beta_i}^\top$
	then 
	\begin{align*}
		\mathbb E[KXK^\top] &= (1-q)^2X + q(1-q)\mathbb E[XL^\top ] + q(1-q)\mathbb E[LX] + q^2 \mathbb E[LXL^\top] \\
		& = (1-q)^2 X + q(1-q)XP^\top + q(1-q)PX + q^2 \mathbb E[LXL^\top] \\
		&= P_q X P_q^\top - q^2PXP^\top + q^2 \mathbb E[LXL^\top].
	\end{align*}
	Next we will compute the term $\mathbb E[LXL^\top] $ as
	\begin{align*}
		\mathbb E[LXL^\top] &= \sum_{i,i'} \mathbb E[e_ie_{\beta_i}^\top Xe_{\beta_{i'}}e_{i'}^\top ] = \sum_{i\neq i'} \mathbb Ex_{\beta_i,\beta_{i'}}e_ie_{i'}^\top + \sum_i \mathbb Ex_{\beta_i,\beta_i} e_ie_i^\top \\
		&= \sum_{\substack{i\neq i'\\ j,j'}} p_{ij}p_{i'j'} x_{jj'} e_ie_{i'}^\top + \sum_i p_{ij} x_{jj} e_i e_i^\top \\
		&= PXP^\top - \sum_{i,j,j'} p_{ij}p_{ij'} x_{jj'} e_{i}e_{i}^\top + \sum_{i,j}p_{ij}x_{jj}e_{i}e_{i}^\top \\          
		& = PXP^\top - \Psi^{-}\Psi(PXP^\top) + \Psi^{-} (P\ \Psi(X)).
	\end{align*}
	Thus putting together the two parts gives
	$$ \Phi(X) = P_qXP_q^\top + q^2\{\Psi^-(P\Psi(X)) -\Psi^- \Psi (PXP^\top)\} $$
	and this concludes the proof.
\end{proof}

In order to obtain a bound on $\tr \{ (I-J)\mathbb E[H(t)H(t)^\top ](I-J) \} $ it is enough to understand $\Phi$, since 
$$ \tr \{ (I-J)\mathbb E[H(t)H(t)^\top](I-J) \}  =  \tr \{(I-J)\Phi^t(I)(I-J)\} ,$$
where $\Phi^t(I)$ denotes the application of $\Phi$ on $I$ $t$ times, i.e. $\underbrace{\Phi\circ \Phi\circ \cdots \circ\Phi}_{t\ times} (I)$. 
\begin{prop} \label{Phi:props} The map $\Phi$ has the following fundamental properties:
	\begin{enumerate}
		\item[(P1)] $\Phi$ is linear,
		\item[(P2)] $\Phi(X^\top) = \Phi(X)^\top$,
		\item[(P3)] for any skew-symmetric matrix $X$, $\Phi(X) = P_qXP_q^\top, $
		\item[(P4)] if $X\geq 0$ then $\Phi(X)\geq 0$, i.e. $\Phi$ keeps the positive semi-definite property,
		\item[(P5)] if $x_{kl}\geq 0$ $\forall (k,l)$, then $\Phi(X)_{kl} \geq 0$  $\forall (k,l)$,
		\item[(P6)]  $J$ is an eigenmatrix of $\Phi$ with eigenvalue $1$, i.e. $\Phi(J) = J$,
		\item[(P7)] for $X\geq 0$, and $P=P^\top$
		we have $\tr\Phi(X) \leq \tr X.$
	\end{enumerate}
	For the adjoint map $\Phi^*$ the following observations can be added:
	\begin{enumerate}
		\item[(P*1)] $\Phi^*(Y) = P_q^\top Y P_q +q^2 \{\Psi^-[P^\top \Psi(Y)] - P^\top  (\Psi^-\Psi Y) P\} $,
		
		\item[(P*2)] if $X\geq 0$ then $\Phi^*(X)\geq 0$, i.e. $\Phi^*$ also keeps the positive semi-definite property,
		\item[(P*3)] if $x_{kl}\geq 0$ $\forall (k,l)$, then $\Phi^*(X)\geq 0$ $\forall (k,l)$,
		
		\item[(P*4)] if $X\mathbf 1 = 0$ then $\Phi^*(X)\mathbf 1=0$.
	\end{enumerate}
\end{prop}

\begin{proof}
	The first three properties follow directly from the definition of $\Phi$, hence their proofs are left to the respected reader. \\
	Property (P4) can be proved as follows. Let $X\geq 0$ and let $w$ be an arbitrary vector, then
	\begin{align*}
		w^\top\Phi(X)w &= w^\top\mathbb E[KXK^\top] w= \mathbb E[w^\top K X K^\top w] \geq 0
	\end{align*}
	(P5) is analogous to (P4), namely
	\begin{align*}
		\Phi(X)_{kl}  = (\mathbb E[KXK^\top])_{kl}  \geq 0.
	\end{align*}
	Property (P6) is the result of a short series of calculations.
	\begin{align*}
		\Phi(J) &= PJP^\top + q^2(\Psi^- P\Psi J - \Psi^-\Psi(PJP^\top)) \\
		& = J + q^2(1/N\cdot I - 1/N\cdot I) = J,
	\end{align*}
	due to the facts $PJ = JP^\top = J$ and $\Psi J =  \mathbf1/N $. \\
	Before proving (P7) let us note that due to $X\geq0$ and the linearity of $\Phi$ it is enough to prove this property for $X=xx^\top$. Using the definition of $K = (1-q)I + q^2 \sum_i e_i e_{\beta_i}^\top$ we have
	\begin{align*}
		\tr\Phi(xx^\top)& = \tr\mathbb E[Kxx^\top K^\top ˘] =\\
		&=\tr \mathbb E\{(1-q)^2 xx^\top + q(1-q)(Lxx^\top + xx^\top L^\top) + q^2Lxx^\top L^\top \} \\
		&= (1-q)^2||x||_2^2 + 2q(1-q) x^\top P x +q^2\mathbb E ||Lx||_2^2 \\
		& = (1-q)^2||x||_2^2 + 2q(1-q) x^\top P x + q^2\sum_{i,j} p_{ij} x_j^2\\
		& \leq  ||x||^2_2
	\end{align*}
	where in the last step we used the facts $P=P^\top$, $P\mathbf 1 = \mathbf 1$ and $x^\top P x \leq \lambda_1(P) ||x||_2^2 = ||x||_2^2 $.
	Now we proceed with proving the properties of the adjoint.\\
	The proof of (P*1) is based on the following series of calculations:
	due to the equivalences
	\begin{align*}
		\tr (\Psi^-(P\ \Psi X ))  Y^\top& =  \sum_{i,k} p_{ik}x_{kk} y_{ii} = \sum_{k} x_{kk} \sum_i p_{ik}y_{ii} = \tr \{X \Psi^-(P^\top Y^\top)\} ,\\
		\tr \Psi^-\Psi(PXP^\top) Y^\top &= \sum_{i,k,l} p_{ik}x_{kl}p_{il}y_{ii} = \sum_{k,l} x_{kl} \sum_i p_{ik}y_{ii}p_{il} = \tr \{X P^\top \Psi(Y^\top) P\}
	\end{align*}
	we have
	\begin{align*}
		\langle \Phi(X), Y \rangle &= \tr \Phi(X)Y^\top = \tr P_qXP_q^\top Y^\top  + q^2\tr \bigg\{ [\Psi^-(P \Psi(X)) ] Y^\top  - [\Psi^-\Psi(PXP^\top)] Y^\top \bigg\} \\
		& = \tr (X P_q^\top Y^\top P_q) + q^2\tr  \{X \Psi^-(P^\top \Psi Y^\top) -X P^\top \Psi (Y^\top)  P\} \\
		& = \langle X,\Phi^*(Y) \rangle.
	\end{align*}
	
	Properties (P*2) and (P*3) can be confirmed analogously to (P4), (P5). \\
	(P*4) is a result of the short derivation
	\begin{align*}
		\Phi^*(X)\mathbf 1 &= P_q^\top XP_q\mathbf 1 + q^2(\Psi^- P \Psi X\mathbf 1 - P^\top (\Psi^-\Psi X) P \mathbf 1) \\
		& = 0 + q^2( \Psi^- (P^\top\Psi(X)) \mathbf 1 - P^\top (\Psi^-\Psi X) \mathbf 1) 
	\end{align*}
	since $P_q\mathbf 1 = P\mathbf 1 = \mathbf 1$ and we assumed $X\mathbf 1 = 0$. For the second term we have 
	\begin{align*}
		(\Psi^- (P^\top \Psi(X)) \mathbf 1)_i &=  \sum_j p_{ji}x_{jj} \\
		(P^\top (\Psi^-\Psi X) \mathbf 1)_i & = \sum_{j} p_{ji} x_{jj}, 	
	\end{align*}
	so $\Phi^*(X)\mathbf 1 = 0$. 
	This concludes the proof.
\end{proof}
\begin{remark}
	Properties (P1), (P2), (P4), (P7) 
	mean that $\Phi$ describes a quantum operation.
\end{remark}

\noindent According to the previously listed properties we have
\begin{cor}
	The cone of positive semi-definite matrices is invariant under the action of $\Phi$. 
\end{cor}
\noindent The following statement is going to help us in proving Theorem \ref{thm:general_rate} as our focus is on the speed of convergence and not the limit $\bar x \mathbf 1$, which is of constant order. 
\begin{cor}
	Since $\tr (I-J)H(t)H(t)^\top (I-J) =\tr H(t)^\top (I-J)H(t) $ and the adjoint of the linear map $ f:X \mapsto AXA^\top$ is the map $f^* : X \mapsto A^\top X A $, it is easy to show that $ \Phi^*(X) = \mathbb E[K^\top X K] $ for any symmetric matrix $X$.
\end{cor} 
\begin{remark} \label{rmk:XJ=0} According to the definition of $K$, we have $KJ=J$, and so 
	$$ (I-J)K(I-J) = (I-J)(K-J) = K-JK = (I-J)K $$
	whence
	$$ \mathbb E[(I-J) KXK^T(I-J)] = \mathbb E[(I-J)K(I-J)X(I-J)K^T (I-J)]. $$
\end{remark}
\begin{cor} \label{corr:Phihat}
	Let us define the operator $\widehat \Phi$ as
	$$\widehat \Phi: X\mapsto (I-J)\Phi(X)(I-J) \in \mathrm{End}(\{Y\in \mathbb R^{N\times N} :  Y=Y^\top, YJ=0\}). $$
	According to properties of $\Phi$ combined with Remark \ref{rmk:XJ=0}, we have 
	$$ \widehat \Phi ^t (X) = (I-J) \Phi^t(X) (I-J), \quad X\in\{Y\in \mathbb R^{N\times N} :  Y=Y^\top, YJ=0\} $$
	furthermore
	\begin{equation}
		\label{eq:phiwithoutIJ}
		\Phi((I-J)X(I-J)) = \Phi(X) + JXP_q^T + P_qXJ -JXJ.
	\end{equation}
\end{cor}

\section{Proofs}
\label{sec:proofs}

The next proposition is the final step before proving Theorem \ref{thm:general_rate}.
\begin{prop}\label{prop:general_rate}
	Let $P$ be a row stochastic matrix. Then
	$$ \tr (\Phi^*)^t(I-J) \leq 
	N \rho\big((I-J)B_q(I-J) \big)^t
	$$ 
	where $B_q = P_q^\top P_q + q^2(\Gamma- P^\top P)$ is a positive definite matrix, and $D$ is a diagonal matrix with $d_{ii} = \sum_jp_{ji}$ on its diagonal.
\end{prop}
\begin{proof}
	Due to its properties $\Phi^* : \mathcal X_0 \to \mathcal X_0 = \{ X \in \mathbb R^{N\times N} : X=X^\top, XJ = 0 \},$ and \\ $X\geq 0 \implies \Phi^*(X) \geq 0$, we have
	\begin{align*} \tr (\Phi^*)^t(I-J) &\leq N \rho((\Phi^*)^t(I-J)) = N \max \{ v^\top (\Phi^*)^t(I-J) v : \|v\| = 1 , v\perp \mathbf 1\} \\
		& = N \| (\Phi^*)^t(I-J)\|_2 \leq N \|(\Phi^*)^t\|_{\mathcal X_0 \to \mathcal X_0} \|(I-J)\|_2 \\ &\leq N \|\Phi^*\|^t_{\mathcal X_0 \to \mathcal X_0}
	\end{align*}
	where 
	$$ \| \Phi^*\|_{\mathcal X_0\to \mathcal X_0} = \max \{ \|\Phi^*(X)\|_2 : \|X\|_2\leq 1, X \in \mathcal X_0\}. $$
	It is not hard to show that $\arg \max \|\Phi^*(X)\|_2 \geq 0$, since let $X = \arg\max \|\Phi^*(X)\|_2$ and let us consider its decomposition $X=X^* - X^-$ with $X^+, X^-\geq 0$. Then
	$$ v^\top \Phi^*(X) v = v^\top \Phi^*(X^+)v - v^\top \Phi^*(X^-)v \leq v^\top \Phi^*(X^+)v + v^\top \Phi^*(X^-) v ,$$
	and this would lead to a contradiction if $X^-$ was not $0$. This implies
	$$ \|\Phi^*\|_{\mathcal X_0 \to \mathcal X_0} = \max \{v^\top \Phi^*(X) v : X\in\mathcal X_0,X\geq 0, \|X\|_2\leq 1, \|v\|\leq 1, v \perp \mathbf 1\}. $$
	meaning that it is enough to bound $v^\top \Phi^*(X) v$ from above. Let $v \perp \mathbf 1$ and $X\in\mathcal X_0$ with $\|X\|_2 \leq 1$, then
	\begin{align*}
		v^\top \Phi^*(X) v &= (P_qv)^\top X P_qv + q^2(\sum_{i,j} p_{ij}x_{ii} v_j^2 - \sum_{i} (Pv)_i^2 x_{ii}).
	\end{align*}
	Due to the conditions imposed on $X$, we have $x_{ii} \in [0,1]$,  furthermore $ \sum_{j}p_{ij} - (Pv)_i^2 \geq 0 $ for any $i$, thus 
	\begin{align*} v^\top \Phi^*(X)v &\leq \|P_qv\|^2 + q^2(\sum_{i,j} p_{ij}v_j^2 - \sum_{i} (Pv)_i^2) \\
		& = v^\top \big(P_q^\top P_q + q^2(\Gamma - P^\top P)\big)v \leq \rho \bigg((I-J)(P^\top_qP_q + q^2(\Gamma-P^\top P)(I-J)\bigg)
	\end{align*} 
	where $\Gamma$ is a diagonal matrix with diagonal elements $\gamma_{ii}=\sum_jp_{ji}$. Note that for symmetric $P$ we have $\Gamma = I$ and so the upper bound above becomes $ (1-q)^2 + 2q(1-q)\lambda_2 + q^2.$
\end{proof}

\noindent With all the tools at our hands we can prove Theorem \ref{thm:general_rate}. 
\begin{proof}[Proof of Theorem \ref{thm:general_rate}]
	According to Lemma 10 in \cite{gerencser2022computable} whose assumptions are clearly satisfied we have
	$$ \limsup_t \frac1t \log \frac1{\min_i w_i(t)} \leq 0 $$
	thus considering the quantity in \eqref{consensus} we can infer that
	$$ \limsup_t \frac1t \log \big(\{\min_i w_i(t)\}^{-1}||x_0||\cdot  ||(I-J) H(t)||_F \big)\leq \limsup_t \frac1t \log ||(I-J)H(t)||_F .$$ The matrix $(I-J)H(t)$ can we written as a product of the i.i.d.\ random matrices $(I-J)K(t)$, therefore due to the Fürstenberg-Kesten theorem it follows that
	\begin{align*}
		\limsup_t \frac1t \log ||(I-J)H(t)||_F &= \lim_t \frac1t \log ||(I-J)H(t)||_F = \lim_t\frac1t\mathbb E\log||(I-J)H(t)||_F \\ 
		& = \lim_t \frac1{2t} \mathbb E \log \| (I-J)H(t) \|_F^2 .
	\end{align*}
	Using Jensen's inequality yields
	$$ \mathbb E \log ||(I-J)H(t)||^2_F \leq \log \mathbb E||(I-J)H(t)||_F^2 $$
	thus 
	$$ \limsup_t \frac1t \log \|(I-J)H(t)\|_F \leq \lim_t \frac1{2t} \log \mathbb E \|(I-J)H(t)\|_F^2.  $$
	By taking the expectation we obtain
	$$ \mathbb E ||(I-J)H(t)||_F^2 = \tr (\Phi^*)^t(I-J). $$
	Combining the series of calculations above with Proposition \ref{prop:general_rate} we arrive at
	$$ \limsup_t \frac1{2t} \log \tr (\Phi^*)^t(I-J)\leq \frac12 \begin{cases}
		\rho\bigg((I-J)(P_q^\top P_q + q^2(\Gamma - P^\top P))(I-J)\bigg) & \mbox{in general,} \\
		(1-q)^2 + 2q(1-q)\lambda_2 + q^2 & \mbox{ if } P=P^\top.
	\end{cases}    $$
	
\end{proof}

Now we will turn to the case when the underlying graph is transitive. In this scenario it is possible to give stronger bounds, but in order to do this we need to reformulate the problem. Rearranging our main quantity of interest as
$$ \mathbb E||(I-J)H(t)||_F^2  = \tr \mathbb E [H(t)^\top(I-J)H(t)], $$
therefore
$$\mathbb E||(I-J)H(t)||_F^2 = \tr \{(\Phi^*)^t(I-J)\}. $$ Let us define $X_t =(\Phi^*)^t(I-J). $ The following two lemmas will help us in our progress.

\begin{lemma}
	Assume that $P$ is a kernel of the transitive Markov chain, implying that it is symmetric and any diagonal element $p^{(k)}_{ii}$ of $P^k$ depends solely on $k$ and not on $i$.  Then we have
	\begin{enumerate}
		\item $X_t \in \mathcal P_t: =\mathrm{Span}\{P^k, J\ ; \ 0\leq k\leq 2t\}$, hence the diagonal of $X_t$ is also constant,
		\item $X_{t+1} = P_q X_t P_q + q^2 r_t (I-P^2) $,
		where $r_t$ denotes the common diagonal element of $X_t$.
	\end{enumerate} 
\end{lemma}
\begin{proof}
	We prove by induction. For $t=0$ $X_0 = I-J$, which trivially is a polynomial of $P$ and $J$. 
	For the induction step $t\to t+1$ assume that $X_t \in \mathcal P_t$, then
	\begin{align*}
		X_{t+1} &= \Phi^*(X_t) = P_qX_tP_q + q^2\{(\Psi^- P \Psi X_t) - P(\Psi^-\Psi X_t )P\}
	\end{align*} 
	since $X_t$ is a polynomial of $P$ and $J$ it is also transitive. Noting $\Psi X_t = r_t \mathbf 1$ and $\Psi^- \Psi X_t = r_t I$, where $r_t$ denotes the common diagonal element of $X_t$, we can derive the recursion
	\begin{align}\label{trans_rec_matrix}
		X_{t+1} = P_qX_tP_q + q^2 r_t(I-P^2)
	\end{align}
	showing that $X_{t+1} \in \mathcal P_{t+1}$.
\end{proof}
\begin{prop}\label{prop:spec_rate}
	If $P$ is symmetric and transitive then $X_t$ and $P$ possess the same eigenvectors. If $v$ is an eigenvector to $P$ and $  X_t$ corresponding to the eigenvalue $\lambda$ and $\mu_t$ respectively then the following recursion holds for $\mu_t$:
	\begin{equation} \label{trans_rec_ev}  \mu_{t+1} = \lambda_q^2 \mu_t + q^2 r_t (1-\lambda^2). 
	\end{equation}
	Recall that $r_t$ is the common diagonal element of $X_t$.
	Furthermore the largest eigenvalue of $X_t$ is asymptotically bounded by the second largest root $\xi_2$ of the polynomial
	$$ p(x) = \prod_i (x-\lambda_{q,i}^2)\bigg(1+\frac{q^2}N \sum_{j>1} \frac{1-\lambda_j^2}{x-\lambda_{q,j}^2}\bigg) $$
	in the following sense: 
	$$ \limsup_t \frac1t \log \max_i \mu_{t,i} \leq \log \xi_2 . $$
\end{prop}
\begin{proof}	
	Using $r_t = \frac 1N \tr X_t = \frac1N \sum_i \mu_{t,i}$ we can write the recursion described in \eqref{trans_rec_ev} as
	\begin{equation} \label{trans_rec} \boldsymbol y_{t+1} = \bigg(D\mathbf + \frac{q^2}N \mathbf b \mathbf 1^\top\bigg)\boldsymbol{y}_t 
	\end{equation}
	with the vectors $(\boldsymbol{y}_t)_{i} = \mu_{t,i} $, $\mathbf b_i = 1-\lambda_{i}^2$, and $D=\dg (\lambda_{q,1}^2,\ldots,\lambda_{q,N}^2)$, where $\lambda_{q,i}$ denotes the $i^{th}$ largest eigenvalue of $P_q$ and $\mu_{t,i}$ denotes the eigenvalue of $X_t$ corresponding to $\lambda_{q,i}$. By correspondence we mean in the sense of defined by the recursion (\ref{trans_rec_ev}), in which case $\mu_{t,1} = 0$, since $\lambda_{q,1} = 1$ and $X_t \mathbf 1 = 0$ for any $t$.
	Let $(\lambda,v)$ denote an eigen-pair of $P$. According to the previous lemma $X_t $ is a polynomial of $P$ and $J$, hence $v$ is an eigenvector of $X_t$, furthermore, due to the symmetry of $P$, we have $JP=PJ=J$. Recursion (\ref{trans_rec_matrix}) then yields
	\begin{align*}
		X_{t+1} v &= P_qX_t P_q v + q^2 r_t(I-P^2)v \\
		& = \lambda_q^2 \mu_t v + q^2 r_t(1-\lambda^2)v	
	\end{align*} 
	and we have 
	$$ \mu_{t+1} = \lambda_q^2\mu_t + q^2 r_t(1-\lambda^2), $$
	proving the first part.
	
	Before continuing with the proof of the second part, let us note that $e_1$ is a left-eigenvector of the matrix $D + q^2 \mathbf b \mathbf 1^\top/N$ corresponding to the eigenvalue 1, meaning that the right-eigenvectors corresponding to a different eigenvalue are orthogonal to $e_1$.  Now let us choose the following vectors as basis: $f_1 = e_1$, $f_i = e_i-e_1, i>1$, then for $i>1$ we have $f_i \perp \mathbf 1$. 
	Writing equation (\ref{trans_rec}) in basis $\{f_i\}$ yields 
	\begin{align*}
		Df_1 &= De_1 = \lambda_{q,1}^2 e_1, \\
		Df_i & = De_i - De_1 = \lambda_{q,i}^2e_i - \lambda_{q,1}^2 e_1 = \\
		& = \lambda_{q,1}^2 f_i +(\lambda_{q,i}^2-\lambda_{q,1}^2)f_1, \quad i>1, \\
		\mathbf b \mathbf 1^\top f_1 &= \mathbf b = b_1 f_1 + \sum_{i>1} b_i(f_i + f_1),  \\
		\mathbf b \mathbf 1^\top f_i&= 0 \quad i>1, 
	\end{align*}
	thus rewriting the matrix $D+q^2 \mathbf b \mathbf 1^\top /N$ in this new basis gives us
	\begin{equation} \label{trans_rec_2} 
		\begin{bmatrix}
			(q^2/N)\sum_i b_i + \lambda_{q,1}^2 & -(\lambda_{q,1}^2-\lambda_{q,2}^2) & -(\lambda_{q,1}^2 -\lambda_{q,3}^2) &\ldots  &  -(\lambda_{q,1}^2 - \lambda_{q,N}^2) \\ 
			(q^2/N) b_2 & \lambda_{q,2}^2 & 0 & \ldots & 0 \\
			(q^2/N) b_3 & 0 & \lambda_{q,3}^2 & \ldots & 0 \\
			\vdots & & & \ddots & \\ 
			(q^2/N)b_N &  0 & \ldots &  0 & \lambda_{q,N}^2 
		\end{bmatrix} 
	\end{equation}
	The characteristic polynomial of the matrix in (\ref{trans_rec_2}) can be computed via expanding along the first column, leading to
	\begin{align*}
		p(x) &= \bigg(x-\frac{q^2}N\sum_i b_i - \lambda_{q,1}^2\bigg)\prod_{i>1}(x-\lambda_{q,i}^2) + \sum_{i>1}\frac{q^2}N b_i(\lambda_{q,1}^2- \lambda_{q,i}^2) \prod_{1<j\neq i}(x-\lambda_{q,j}^2) \\
		& = \prod_{i>1} (x-\lambda_{q,i}^2) \bigg\{x - \lambda_{q,1}^2-\frac{q^2}N\sum_i (1-\lambda_{i}^2)   + \frac{q^2}N \sum_{j>1}(1-\lambda_j^2)\frac{(1 - \lambda_{q,j}^2)}{x-\lambda_{q,j}^2} \bigg\} .
	\end{align*}
	Exploiting the fact $\lambda_1 = \lambda_{q,1}= 1$ we obtain
	\begin{align*}  p(x)&= \prod_{i>1}(x-\lambda_{q,i}^2)\bigg\{ x-1 + \frac {q^2}{N}\sum_{j>1}\bigg[ (1-\lambda_j^2)\bigg(\frac{1-\lambda_{q,j}^2}{x-\lambda_{q,j}^2} -1\bigg) \bigg] \bigg\} \\
		& = \prod_{i>1} (x-\lambda_{q,i}^2) \bigg\{x-1 + \frac {q^2}N \sum_{j>1} (1-\lambda_j^2)\frac{1-x}{x-\lambda_{q,j}^2} \bigg\}
		\\ &= \prod_i (x-\lambda_{q,i}^2)\bigg(1-\frac{q^2}N \sum_{j>1} \frac{1-\lambda_j^2}{x-\lambda_{q,j}^2}\bigg). 
	\end{align*}
	We note here that due the initialization $\mu_0 = (0,1,\ldots, 1)$ and the fact that $e_1$ is a left-eigenvector of the recursion, we have for each subsequent vector $\mu_t \perp e_1$. This means that we are only interested in the second largest root $\eta_2$ of the characteristic polynomial to obtain the asymptotic growth rate of $\mu_t$, i.e. 
	$$ ||\mu_t||_{\infty} \leq C\eta_2^t\ ||\mu_0||_{\infty} $$
	for some $C>0$.
\end{proof}      
\begin{proof}[Proof of Theorem \ref{thm:special_rate}]
	The first part of the proof, up until the point where we have to bound the quantity $\mathbb E ||(I-J)H(t)||_F^2$ from above, is analogous to the proof of Theorem \ref{thm:general_rate}, hence it will be omitted here. Before proceeding with the actual proof we note that 
	$$ ||(I-J)H(t)||_F^2 = \tr (I-J)H(t)H(t)^\top(I-J) = \tr H(t)^\top(I-J)H(t), $$
	whence
	$$ \mathbb E ||(I-J)H(t)||_F^2 = \tr \{(\Phi^*)^t(I-J)\} = \sum_j\mu_{t,j}\leq N \max_j \mu_{t,j},$$
	where $\mu_{t,j}$ denotes the eigenvalue of $X_t = (\Phi^*)^t(I-J)$ corresponding to the $j^{th}$ largest eigenvalue $\lambda_j$ of $P$. Due to $X_0 = I-J$ and $\lambda_1= 1$ with $v_1 = c \mathbf 1$, we have $\mu_{0,1} = 0$ and due to the recursion \eqref{trans_rec_ev}, $\mu_{t,1} = 0$ for any $t>0$. In Proposition \ref{prop:spec_rate} it was shown that 
	$$ ||\mu_t||_{\infty} \leq C \eta_2^t ||\mu_0|| = C \eta_2^t, $$ where $\eta_2$ denoted the second largest root of the polynomial
	\begin{align*}
		p(x) = \prod_{i} (x-\lambda_{q,i}^2) - \frac{q^2}N \sum_{i>1} (1-\lambda_i^2)\prod_{i\neq j>1} (x-\lambda_{q,i}^2).
	\end{align*}
	Exploiting the fact that $1$ is a root of $p$ we have
	\begin{align*}
		p_1(x) := p(x)/(x-1)=\prod_{i>1} (x-\lambda_{q,i}^2) - \frac{q^2}N \sum_{i>1} (1-\lambda_i^2)\prod_{i\neq j>1} (x-\lambda_{q,i}^2)
	\end{align*}
	thus by denoting $\xi_1$ the largest root of $p_1(x)$ we can write $||\mu_t||_\infty \leq C \xi_1^t$. Altogether
	$$ \limsup_t \frac1{2t}\log\mathbb E||(I-J)H(t)||_F^2 \leq \frac12 \log \xi_1 , $$
	and this concludes the proof.
\end{proof}

\section{Numerical experiments}
\label{sec:sim}

In order to evaluate the performance of the bounds obtained on the convergence rate, we perform a detailed numerical comparison of the available quantities. For setups with $N\leq 120$ the synchronous gossip process is realized for $t=500$ steps and the approximate rate is expressed as
$$\frac{1}{t}\log \bigg\|\frac{1}{\sqrt{N}} (I-J)H(t)  \Psi^-(w(t))^{-1} \bigg\|_F.$$
For $N>120$ we set $t=1000$ and, for complexity and memory usage reduction, we use the modified approximation
$$\frac{1}{t}\log \bigg\|\frac{1}{\sqrt{M}}  X^M(0)H(t)  \Psi^-(w(t))^{-1} \bigg\|_F$$
where $X^M(0)$ is an $M\times N$ matrix of uniform random independent rows in $1^\perp$ with unit norm, representing various initializations, reaching the principal rate with high probability. We choose $M=\lfloor \sqrt{N} \rfloor$.

This is compared both with $\eta := \frac 12 \log \rho((I-J)^{\otimes 2} \mE (K(1)^{\otimes 2}))$ which is the bound of \cite{gerencser2022computable} and the bounds developed in the current paper.

For the context of Theorem \ref{thm:general_rate}, more precisely its specialized version Corollary \ref{cor:general_rate_graph} we build a Barabási-Albert random graph with 2 edges added with each new vertex, and assign uniform probabilities for choosing gossip recipients. See Figure \ref{fig:sim_general} for the resulting approximate rates and bounds as $q$ varies in $(0,1)$. We can see that $\eta$ provides a very close fit, the current bound is only moderately usable for small $q$. However, note that for $N=100$ handling the $N^2\times N^2$ matrices needed for the tensor product based bound is getting computationally heavy, thus we only plot the simulations versus the current bound. We also conducted simulation experiments for $N=5000$ in which case our bound was mostly nonnegative and hence did not carry any useful information.
\begin{figure}[t]
	\centering
	\begin{subfigure}[t]{0.4\textwidth}
		\includegraphics[width=\textwidth]{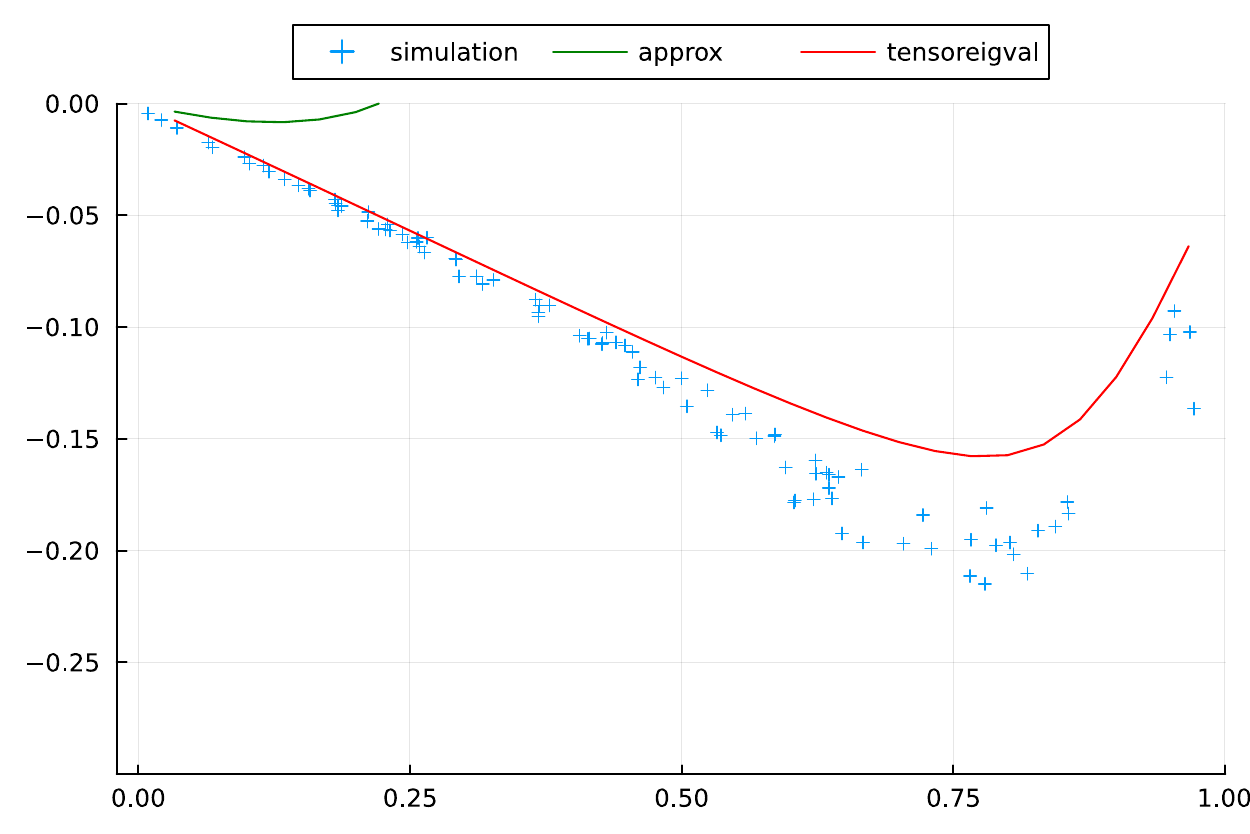}
		\caption{$N=24$ and $k=2$.}
	\end{subfigure} \hfill 
	\begin{subfigure}[t]{0.4\textwidth}
		\centering 
		\includegraphics[width=\textwidth]{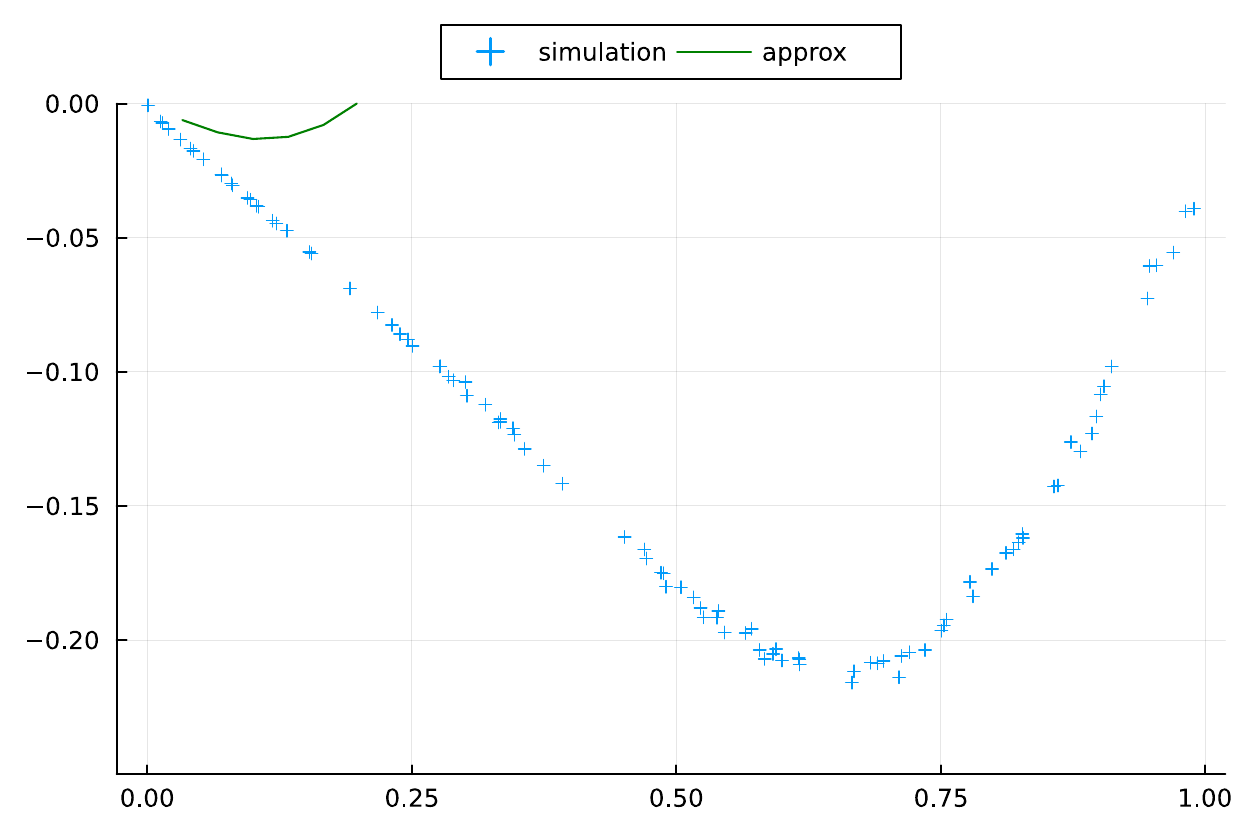}
		\caption{$N=100$ and $k=2$}
	\end{subfigure}
	\caption{Simulations and bounds \eqref{eq:general_rate} for general graph without symmetries, generated by the Barabási-Albert algorithm.}
	\label{fig:sim_general}
\end{figure}

In the case of symmetric $P$, the second part of Theorem \ref{thm:general_rate} provides a more usable bound as shown in Figure \ref{fig:sim_symmetric}. We use random regular graphs and uniform recipient probabilities again. The resulting bound captures the linear trend when $q$ is near 0, it is a factor $\approx 2$ off from the numerical value when $q\le 0.5$ and then deteriorates.

\begin{figure}[h]
	\centering
	\begin{subfigure}[t]{0.32\textwidth}
		\centering
		\includegraphics[width=\textwidth]{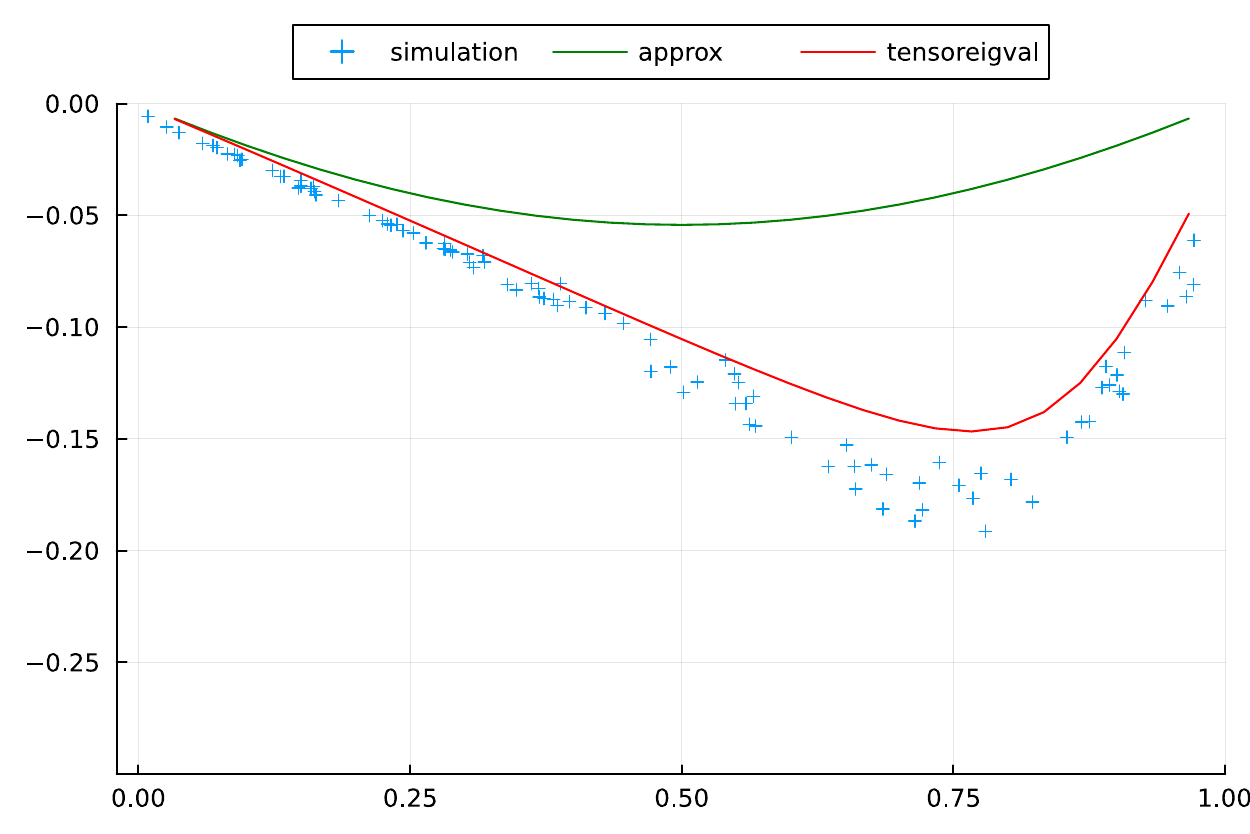}
		\caption{$N=24$ and $d=4$.}
	\end{subfigure} 
	\begin{subfigure}[t]{0.32\textwidth}
		\centering 
		\includegraphics[width=\textwidth]{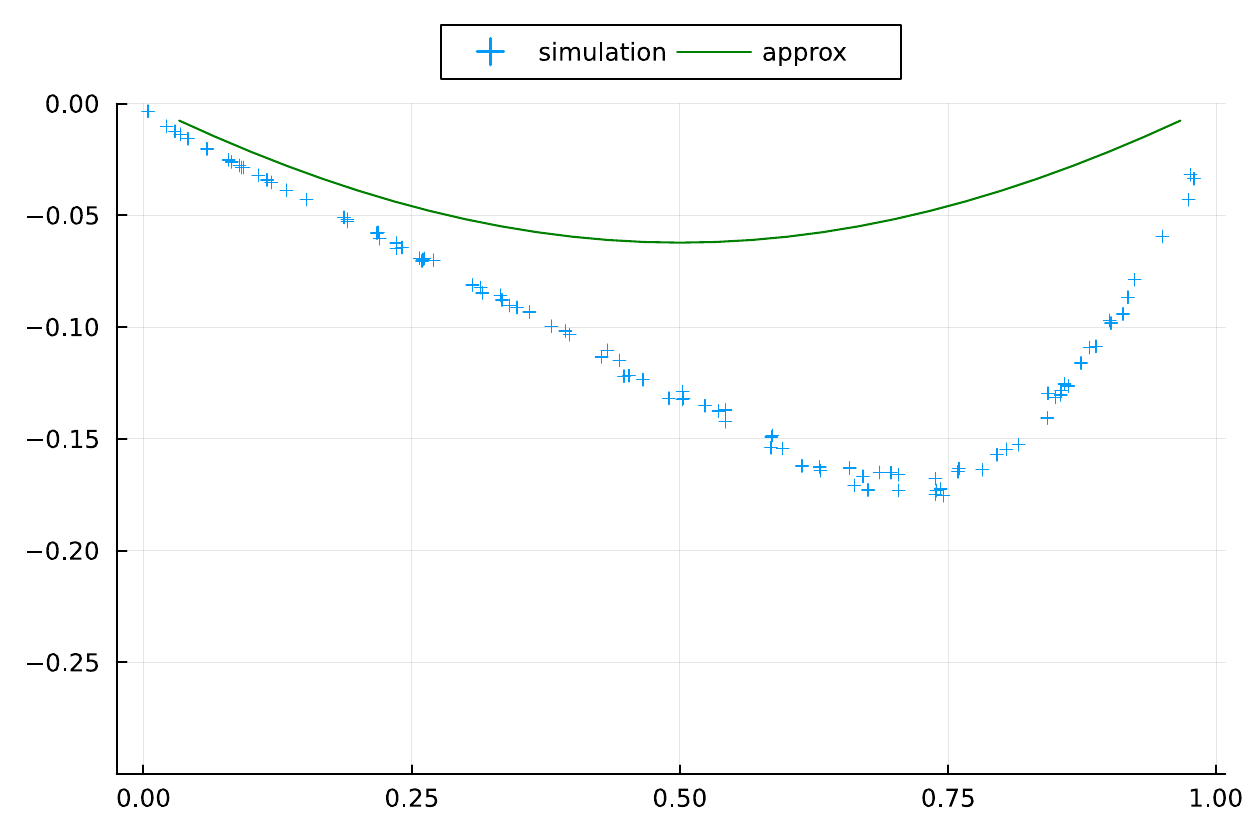}
		\caption{$N=100$ and $d=5$.}
	\end{subfigure}
	\begin{subfigure}[t]{0.32\textwidth}
		\centering
		\includegraphics[width=\textwidth]{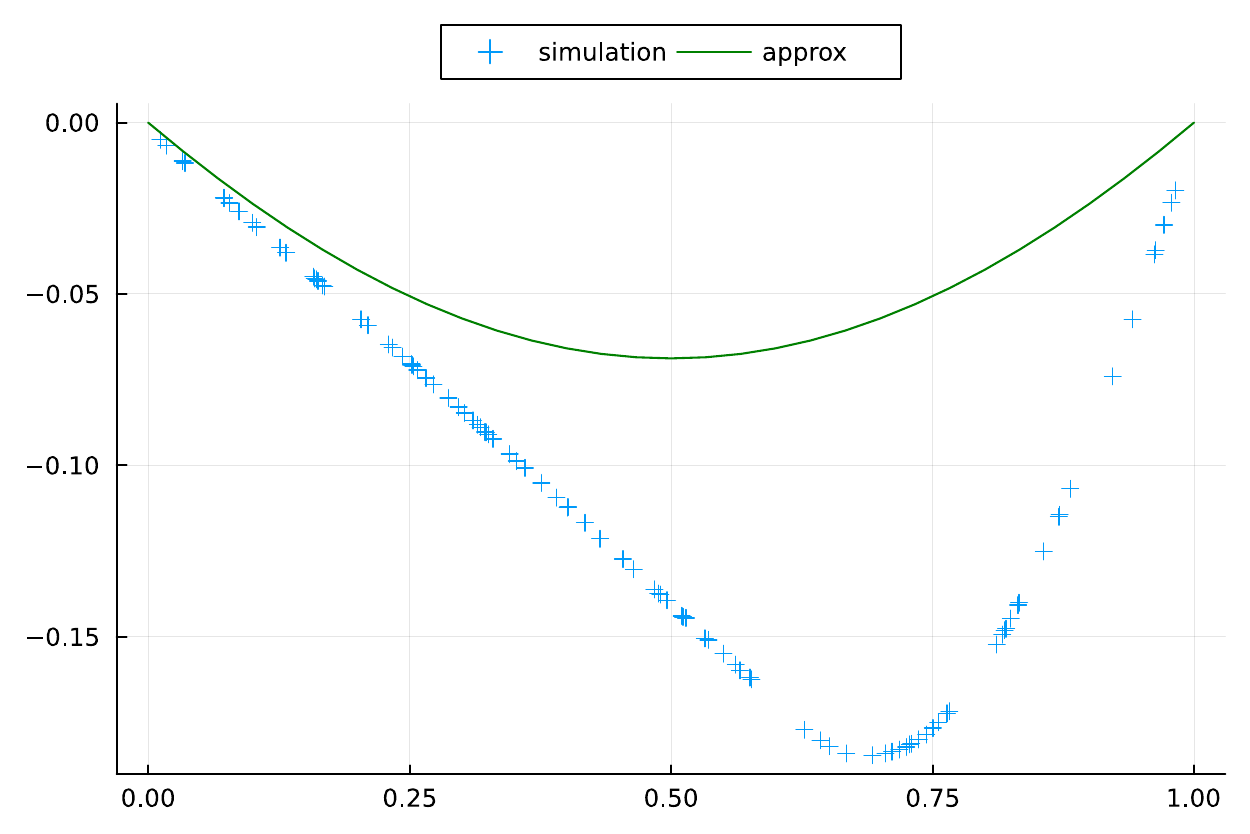}
		\caption{$N=5000$ and $d=6$.}
	\end{subfigure}
	\caption{Simulations and bounds \eqref{eq:general_rate_symmetric} for random regular graphs.}
	\label{fig:sim_symmetric}
\end{figure}

For the more refined bound of Theorem \ref{thm:special_rate}, we generate a random Cayley-graph of $S_k$. This time we see a substantially better fit as shown in Figure \ref{fig:sim_transitive}. In fact, it recovers $\eta$ which is expected from the exact analysis carried out in its proof, but circumvents the necessity of working with the large matrices of the Kronecker-product.

\begin{figure}[h!]
	\centering
	\begin{subfigure}[t]{0.32\textwidth}
		\includegraphics[width=\textwidth]{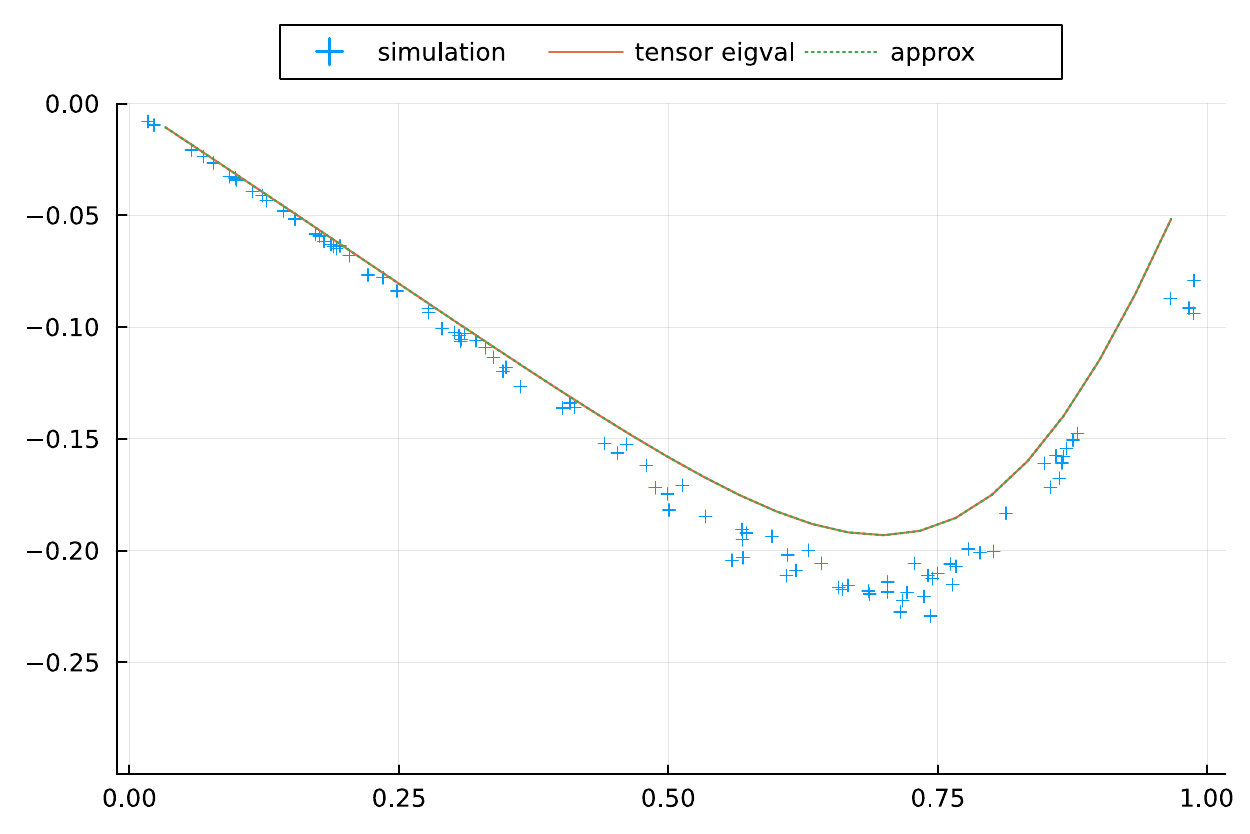}
		\caption{$S_4$ with 2 generators \\ $N=24, d= 4$}
	\end{subfigure} 
	\begin{subfigure}[t]{0.32\textwidth}
		\centering 
		\includegraphics[width=\textwidth]{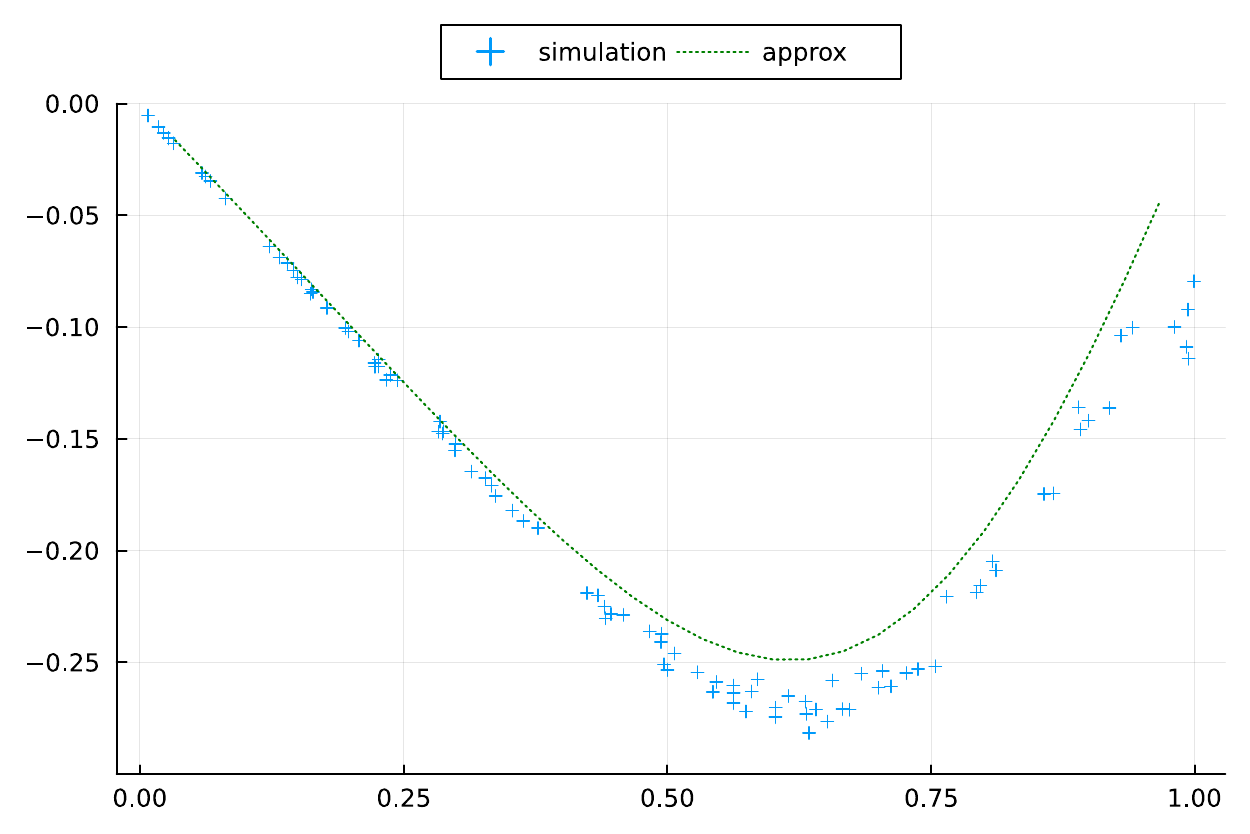}
		\caption{$S_5$ with 3 generators \\ $N=120, d=5$}
	\end{subfigure}
	\begin{subfigure}[t]{0.32\textwidth}
		\centering 
		\includegraphics[width=\textwidth]{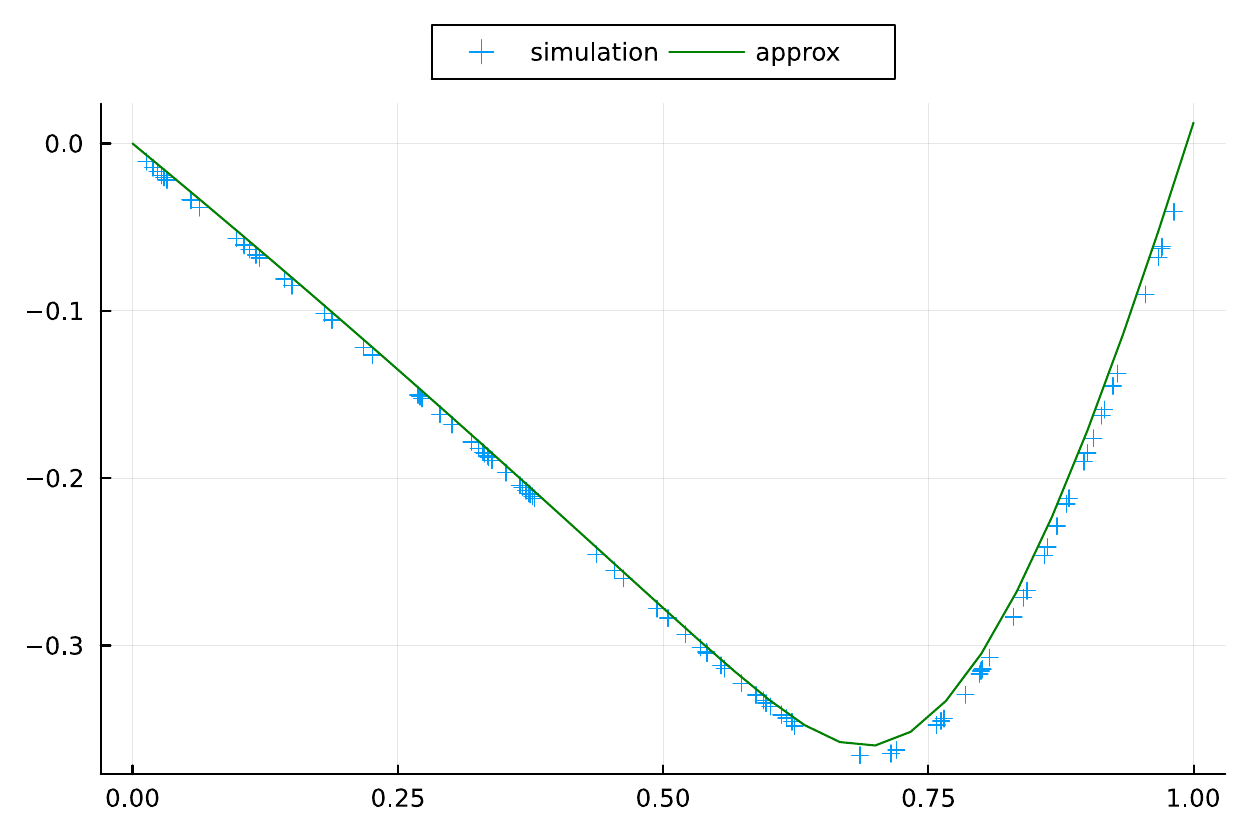}
		\caption{$ S_7$ with 3 generators,\\ $N=5040$, $d=6$}
	\end{subfigure}
	\caption{Simulations and bounds \eqref{eq:transitive_rate} for random Cayley-graphs.}
	\label{fig:sim_transitive}
\end{figure}

\newpage

\section{Summary and future plans}\label{sec:conc}

In this work we have presented bounds of various accuracy depending on the level of symmetry of the underlying topology.
Along the way proving our main results we have developed a framework as described in Section \ref{sec:tools} relying on matrix operators $\Phi, \Phi^*$ that we hope can be useful when analyzing similar dynamics. 

The computational cost of these bounds is orders of magnitude less than that of the simulations or computation of $\eta$ from \cite{gerencser2022computable} confirming their usefulness in assessing the efficiency of various push-sum algorithms.

Concerning our future plans, we observed in some of our numerical experiments that the expression in \eqref{eq:transitive_rate} was also a valid bound for regular graphs. This lead us to state the following
\begin{conjecture}
	The bound given in \ref{thm:special_rate} remains valid in the simple symmetric case, e.g. for regular graphs without transitivity.
\end{conjecture}

\bibliographystyle{siam}
\bibliography{pushsum,generic}

\end{document}